\documentclass[12pt]{amsart}
\usepackage{amsfonts}
\usepackage{amssymb,amstext,amscd,amsmath,mathdots}
\usepackage{fullpage}
\usepackage[all]{xy}
\usepackage[usenames,dvipsnames]{color}

\usepackage{graphicx}
\usepackage{bbding}
\usepackage[misc]{ifsym}

\numberwithin{equation}{section}
\theoremstyle{plain}
\newtheorem{thm}{Theorem}[section]
\newtheorem{cor}[thm]{Corollary}
\newtheorem{prop}[thm]{Proposition}
\newtheorem{lem}[thm]{Lemma}

\theoremstyle{definition}
\newtheorem{rem}[thm]{Remark}

\newtheorem{defn}[thm]{Definition}

\newtheorem{ques}{Question}

\renewcommand{\P}{{\mathcal{P}}}

\renewcommand{\phi}{\varphi}
\newcommand{\upchi}{{\raise.35ex\hbox{$\chi$}}}

\begin{document}

\title{Power Set of Some Quasinilpotent Weighted shifts on $\ell^p$}

\author[C. L. Hu]{ChaoLong Hu}
\address{ChaoLong HU: School of Mathematics\\Jilin University\\Changchun 130012\\P.R. CHINA}\email{huzl21@mails.jlu.edu.cn}

\author[Y. Q. Ji ]{YouQing Ji }
\address{YouQing JI: School of Mathematics\\Jilin University\\Changchun 130012\\P.R. CHINA}\email{jiyq@jlu.edu.cn}

\thanks{Supported by NNSF of China (Grant No.12271202 and No.12031002).}
\begin{abstract}
  For a quasinilpotent operator $T$ on a Banach space $X$, Douglas and Yang defined $k_x=\limsup\limits_{z\rightarrow 0}\frac{\ln\|(z-T)^{-1}x\|}{\ln\|(z-T)^{-1}\|}$ for each nonzero vector $x\in X$, and call $\Lambda(T)=\{k_x: x\ne 0\}$ the {\em power set} of $T$.
  They proved that the power set have a close link with $T$'s lattice of hyperinvariant subspaces.
  In this paper, we computes the power set of some weighted shifts on $\ell^p$ for $1\leq p< \infty$.
   The following results are obtained:
  (1) If $T$ is an injective quasinilpotent forward unilateral weighted shift on $\ell^p(\mathbb{N})$, then
      $\Lambda(T)=\{1\}$ when $k_{e_0}=1$, where $\{e_n\}_{n=0}^{\infty}$ be the canonical basis for $\ell^p(\mathbb{N})$;
  (2) There is a class of backward unilateral weighted shifts on $\ell^p(\mathbb{N})$ whose power set is
      $[0,1]$;
  (3) There exists a bilateral weighted shift on $\ell^p(\mathbb{Z})$ with power set $[\frac{1}{2},1]$.
\end{abstract}
\maketitle


\section{Introduction}

Let $X$ be a complex Banach space, and $\mathcal{B}(X)$ the algebra of bounded linear operators on $X$. For $T\in \mathcal{B}(X)$,
let $\sigma(T)$ be the spectrum of $T$.
Say $T$ is {\em quasinilpotent} if $\sigma(T)=\{0\}$.
The hyperinvariant subspace problem asks:
Does every bounded operator $T$ on an infinite dimensional Hilbert space have a non-trivial closed hyperinvariant (i.e., invariant for all the operators commuting with $T$) subspace?
This question is still open, especially for quasinilpotent operators.
We refer the readers to \cite{FJKP05,FP05, RR73, Read97, Sar74} and the references therein for more information along this line.

The {\em power set $\Lambda(T)$} of a quasinilpotent operator $T$ was introduced by  Douglas and Yang (see \cite{DY2, DY3}).
For reader's convenience, we recall it below.
\begin{defn}
Suppose that $T\in\mathcal{B}(X)$ is quasinilpotent and  $x\in X\setminus\{0\}$. Let
$$k_{(T,x)}=\limsup\limits_{z\rightarrow 0}\frac{\ln\|(z-T)^{-1}x\|}{\ln\|(z-T)^{-1}\|}.$$
Usually, we briefly write $k_x$ instead of $k_{(T,x)}$ if there is no confusion possible.
Set $\Lambda(T)=\{k_x : x\ne 0\}$, and call it the {\em power set} of $T$.
\end{defn}
\begin{rem}
For $\Lambda(T)$, there are two basic and often-used  facts.
\end{rem}
\begin{itemize}
  \item [(1)] From 
  $$\frac{\|x\|}{|z|+\|T\|}\leq\|(z-T)^{-1}x\|\leq\|(z-T)^{-1}\|\|x\|,~~ \forall z\ne 0,$$
  it follows that $\Lambda(T)\subset [0,1]$.
  \item [(2)]
  Since $\lim\limits_{z\rightarrow 0}\|(z-T)^{-1}\|=+\infty$ when $\sigma(T)=\{0\}$,
  it holds that $k_{(\alpha x)}=k_x$ for all $\alpha\ne 0$ and $x\ne 0$.
\end{itemize}
\begin{prop}[see \cite{DY3}]
Suppose that $T\in{\mathcal{B}(\mathcal{H})}$, if $T^n=0$ and $T^{n-1}\neq0$, then $\Lambda(T)=\{\frac{j}{n}:j=1,2,\cdots,n\}$.
\end{prop}

Following from that proposition above, when $T$ is Jordan block of order $n$, it is obvious that $\operatorname{Lat}T$ is order isomorphic $\{0\}\bigcup \Lambda(T)$.
Therefore, what instinct information about $\operatorname{Lat}T$ can be seen from $\Lambda(T)$?
Douglas and Yang established a link between the power set and the hyperinvariant subspace problem as follows.
\begin{prop}[see {\cite[Proposition 7.1, Corollary 7.2]{DY3}}]\label{pthy}
Let $T\in \mathcal{B}(X)$ be a quasinilpotent operator.
For $\tau\in[0,1]$, write $M_\tau=\{x : k_{x}\leq\tau\}\bigcup\{0\}$.
Then, $M_\tau$ is a linear subspace of $X$, and $A(M_\tau)\subset M_\tau$ for every $A$ commuting with $T$.

In particular, when $\Lambda(T)$ contains two different points $\tau$ with closed $M_\tau$,
$T$ has a nontrivial closed hyperinvariant subspace.
\end{prop}

In fact, $M_\tau$ is not always closed.
In \cite{JiLiu},
Ji and Liu constructed a quasinilpotent operator $T\in{\mathcal{B}(\mathcal{H})}$ with power set $[0,1]$,
for which $M_\tau$ is not closed for each $\tau\in [0,1)$.
But they proved that there exists a subset $\mathcal{N}$ of $\operatorname{Lat}T$,
which is order isomorphic to $\Lambda(T)$.
Therefore, the following question which appears in \cite{JiLiu}, comes as most natural.
\begin{ques}\label{Q:oriso}
For any quasinilpotent operator $T\in\mathcal{B}(X)$, can one find a subset $\mathcal{N}$ of $\operatorname{Lat}T$ which is order isomorphic to $\Lambda(T)$?
\end{ques}
Furthermore, we can consider more questions about the power set.
 In \cite{HeZhu, JiLiu}, we observe that the subsets $\{1\}$ and $[0,1]$ can be the power set of some quasinilpotent operator.
Can we find other subsets of $[0,1]$ that correspond to power set of some quasinilpotent operator?
Thus, we may consider the following question.
\begin{ques}\label{Q:subset}
Which subsets of $[0,1]$ could be the power set of a quasinilpotent operator $T\in \mathcal{B}(X)$?
\end{ques}

Given a Banach space $X$, we let $X^\prime$ denote its dual space. Suppose that $T\in\mathcal{B}(X)$, we denote by $T^\prime$ the adjoint of acting on $X^\prime$. Comparing with the fact that $\sigma(T)=\sigma(T^\prime)$, one may naturally consider the following question.

\begin{ques}\label{Q:adjoint}
Are $\Lambda(T)$ and $\Lambda(T^\prime)$ intrinsically linked?
\end{ques}
In light of these questions, this paper is devoted to the calculations of the power set of quasinilpotent weighted shifts on $\ell^p$.
To proceed, we recall some terminology and notation.

For $1\leq p<\infty$, let $\ell^p(\mathbb{Z})$ be the Banach space of all two-sided $p$-summable sequences of complex numbers, and $\ell^p(\mathbb{N})$ be the Banach space of all sequences
$\{x_n\}_{n\in \mathbb{Z}}\in \ell^p(\mathbb{Z})$ for which $x_n=0$ for all $n\leq0$.
Let $\{e_k\}_{k\in\mathbb{K}}$ be the canonical basis for $l^p(\mathbb{K})$, where the set $\mathbb{K}$ will be the set $\mathbb{N}$ or $\mathbb{Z}$.

Say operator $T$ on $\ell^p(\mathbb{N})$ is a {\em forward unilateral weighted shift} with weight sequence $\{w_n\}_{n=0}^\infty$ if $T\in \mathcal{B}(\ell^p(\mathbb{N}))$ and if
$$Te_n=w_ne_{n+1},\quad \forall n\in \mathbb{N},$$
i.e., $T$ has the following matrix representation
\[
\begin{bmatrix}
  0  \\
  w_0&0\\
  &w_1&0\\
  &&w_2&0\\
  &&&\ddots&\ddots
\end{bmatrix}
\begin{matrix}
e_0\\
e_1\\
e_2\\
e_3\\
\vdots
\end{matrix}
.\]

Say operator $T$ on $\ell^p(\mathbb{N})$ is a {\em backward unilateral weighted shift} with weight sequence $\{w_n\}_{n=0}^\infty$ if $T\in \mathcal{B}(\ell^p(\mathbb{N}))$ and if
$$Te_0=0,  \quad Te_n=w_{n-1}e_{n-1}, \quad \forall n\in \mathbb{N},$$
i.e., $T$ has the following matrix representation
\[
\begin{bmatrix}
  0&w_0 \\
  &0&w_1\\
  &&0&w_2\\
  &&&\ddots&\ddots
\end{bmatrix}
\begin{matrix}
e_0\\
e_{1}\\
e_{2}\\
\vdots
\end{matrix}
.\]

Say operator $T$ on $\ell^p(\mathbb{N})$ is a {\em bilateral weighted shift} with weight sequence $\{w_n\}_{n=-\infty}^{+\infty}$ if $T\in\mathcal{B}(\ell^p(\mathbb{Z}))$ and if
$$Te_n=w_ne_{n+1}, \quad \forall n\in\mathbb{Z},$$
i.e., $T$ has the following matrix representation
\[
\begin{bmatrix}
\ddots\\
\ddots&0&  \\
  &w_{-2}&0\\
  & &w_{-1}&0\\
  & & &w_0&0\\
  & & & &w_1&0\\
  & & & & &\ddots&\ddots\\
\end{bmatrix}
\begin{matrix}
\vdots\\
e_{-2}\\
e_{-1}\\
e_0\\
e_1\\
e_2\\
\vdots
\end{matrix}
.\]

Now we review some known results about the calculations of the power set of quasinilpotent operators.
Let $V$ be the Volterra operator on the Hardy space $H^2(\mathbb{D})$, defined by $(Vf)(z)=\int_{0}^{z}f(\xi)d\xi, f\in H^2(\mathbb{D})$.
Liang and Yang \cite{LY} first observed that $\Lambda(V)=\{1\}$.
Actually, $V$ is a strongly strictly cyclic forward unilateral weighted shift on $H^2(\mathbb{D})$ with weight sequence $\{\frac{1}{n+1}\}_{n=0}^{\infty}$.
In \cite{JiLiu}, Ji and Liu showed that if $T$ is a strongly strictly cyclic quasinilpotent forward unilateral weighted shift on $\ell^2(\mathbb{N})$, then $\Lambda(T)=\{1\}.$
Later,
He and Zhu generalized Ji and Liu's above mentioned result by showing that 
if $T$ is either a strictly cyclic quasinilpotent forward unilateral weighted shift
or a strongly strictly cyclic quasinilpotent operator
$l^2(\mathbb{N})$, then $\Lambda(T)=\{1\}$ \cite{HeZhu}.

Given a weight sequence $\{w_n\}_n$, for $1\leq p<\infty$,
let $T_p$ be the weighted shift with weight sequence $\{w_n\}_n$ on $\ell^p$.
A natural question is:
\begin{ques}\label{Q:inch}
Is the power set $\Lambda(T_p)$ independent of $p$?
\end{ques}
In Section $2$, we prove that if $T_p$ is a quasinilpotent forward unilateral weighted shift on $\ell^p(\mathbb{N})$,
then $\Lambda(T_p)=\{1\}$ when $k_{(T_p,e_0)}=1$.
As an application, we show that if $T_p$ is a strictly cyclic quasinilpotent forward unilateral weighted shift on $\ell^p(\mathbb{N})$, then $\Lambda(T_p)$ are the same for different $p$ and $\Lambda(T_p)=\{1\}$.
This also gives a partial answer to Question \ref{Q:inch}.

Next, we consider the backward weighted shift on $\ell^p(\mathbb{N})$.
Indeed, given a forward weighted shift $A$ on $\ell^p(\mathbb{N})$ with weight sequence $\{w_n\}_{n=0}^{\infty}$,
it is well-known that $A^{\prime}$ is the backward weighted shift on $(l^p(\mathbb{N}))^{\prime}$
with weight sequence $\{w_n\}_{n=0}^{\infty}$.
Now we suppose that $T$ is a quasinilpotent backward unilateral weighted shift on $\ell^p(\mathbb{N})$.
Since $\|(z-T)^{-1}e_n\|_p^p$ is a polynomial in $\frac{1}{z}$ with degree $n+1$, it is not difficult to check that $k_{(T,e_n)}=0$ for each $n\geq0$. 
Moreover, one can also check that $k_{(T,x)}=0$ when $x\in \operatorname{span}\{e_n: n\geq 0\}$.
Thus, it is interesting to consider the following question.
\begin{ques}\label{Q:backshift}
For any quasinilpotent backward unilateral weighted shift $T$ on $\ell^p(\mathbb{N})$, is it true that $\Lambda(T)=\{0\}$?
\end{ques}
In Section $3$, we first prove that the power set of a class of backward unilateral weighted shifts 
can contain $1$. That provides a negative answer to Question \ref{Q:backshift}.
Moreover, we show that the power set of some backward unilateral weighted shifts is $[0,1]$.
In particular, they include the backward unilateral weighted shift $T$ with weight sequence $\{\frac{1}{n+1}\}_{n=0}^{\infty}$.
In fact, $T$ is unicellular and has only countable many invariant subspaces (see \cite{Shs}).
So one can not find a subset $\mathcal{N}$ of $\operatorname{Lat} T$ which is order isomorphic to $\Lambda(T)$.
That also provides a negative answer to Question \ref{Q:oriso}.
In addition, in view of these results, it also provides some information for Question \ref{Q:adjoint}.

In Section $4$, we construct a bilateral weighted shift on $\ell^p(\mathbb{Z})$ for $1< p<\infty$, 
which power set is $[\frac{1}{2},1]$.
This provides a new method to construct a quasinilpotent operator whose power set is an interval segment on $[0,1]$, and brings a new insight into the Question \ref{Q:subset}.

\section{Power set of forward unilateral weighted shifts}
In this section, we aim to investigate the power set of forward unilateral weighted shifts on $\ell^p(\mathbb{N})$ for $1\leq p<\infty$.
Before proceeding, let us recall some basic properties for weighted shifts.

Given a weighted shift $T$ with the weight sequence $\{w_n\}_{n}$, there exists a surjective isometry $U$ on $\ell^p(\mathbb{N})$ such that $U^{-1}TU$ is the weighted shift with weight sequence $\{|w_n|\}_{n}$ (see \cite{Shs}). From this, we say $T$ have circular symmetry. Moreover, if $w_n\neq 0$ for each $n$, then say that $T$ is injective.
Thus, without loss of generality, we may assume that $w_n>0$ for all $n$.
In order to prove the main theorem, we first give the following lemmas.

\begin{lem}\label{F:poly0}
Let $T$ be a quasinilpotent forward unilateral weighted shift on $\ell^p(\mathbb{N})$.
Then, for any polynomial $\varphi$,
$$\lim\limits_{z\rightarrow 0}\frac{\ln|\varphi(|z|)|}{\ln\|(z-T)^{-1}e_0\|_p}=0.$$
\end{lem}

\begin{proof}
One can refer to the proof in \cite[Lemma 3.6]{JiLiu}.
\end{proof}

\begin{lem}\label{F:Tx=x}
Let $T$ be a quasinilpotent forward unilateral weighted shift on $\ell^p(\mathbb{N})$.
Suppose that $x\in \ell^p(\mathbb{N})$ and $k_x\neq 0$, then $k_{(T^nx)}=k_{x}$ for all $n\geq 0$.
\end{lem}
\begin{proof}
By Lemma \ref{F:poly0}, it is easy to check that
$$\lim\limits_{z\rightarrow 0}\frac{\ln\frac{1}{|z|}}{\ln\|(z-T)^{-1}\|_p}=0.$$
For $z\neq0$, we have 
\begin{align*}
(z-T)^{-1}x&=\frac{1}{z}+\frac{1}{z^2}Tx+\frac{1}{z^3}T^2x+\frac{1}{z^4}T^3x+\cdots\\
&=\frac{1}{z}+\frac{1}{z}(\frac{1}{z}Tx+\frac{1}{z^2}T^2x+\frac{1}{z^3}T^3x+\cdots)\\
&=\frac{1}{z}+\frac{1}{z}(z-T)^{-1}Tx.
\end{align*}
Therefore,
$\frac{1}{|z|}\|(z-T)^{-1}Tx\|_p=\|(z-T)^{-1}x-\frac{1}{z}\|_p$ and 
$$\ln\frac{1}{|z|}+\ln\|(z-T)^{-1}Tx\|_p=\ln\|(z-T)^{-1}x-\frac{1}{z}\|_p.$$
Since $\lim\limits_{z\rightarrow 0}\|(z-T)^{-1}\|=+\infty$ and $k_x\neq 0$, we can obtain that 
$k_{(Tx)}=k_{x}$. Assume that $n=k$, have $k_{(T^kx)}=k_{x}$, we next prove that $k_{(T^{k+1}x)}=k_{x}$.
For $z\neq0$, we also have
\begin{align*}
(z-T)^{-1}T^kx&=\frac{1}{z}T^kx+\frac{1}{z^2}T^{k+1}x+\frac{1}{z^3}T^{k+2}x+\frac{1}{z^4}T^{k+3}x+\cdots\\
&=\frac{1}{z}T^kx+\frac{1}{z}(\frac{1}{z}T^{k+1}x+\frac{1}{z^2}T^{k+2}x+\frac{1}{z^3}T^{k+3}x+\cdots)\\
&=\frac{1}{z}T^kx+\frac{1}{z}(z-T)^{-1}T^{k+1}x.
\end{align*}
Thus,
$\frac{1}{|z|}\|(z-T)^{-1}T^{k+1}x\|_p=\|(z-T)^{-1}x-\frac{1}{z}T^kx\|_p$ and
$$\ln\frac{1}{|z|}+\ln\|(z-T)^{-1}T^{k+1}x\|_p=\ln\|(z-T)^{-1}T^kx-\frac{1}{z}T^kx\|_p.$$
It follows that $k_{(T^{k+1}x)}=k_{x}$.
So, $k_{(T^nx)}=k_{x}$ for all $n\geq 0$.
\end{proof}


For $k\geq 0$, set $f_k(x)=x_k$, $\forall x=\{x_n\}_{n=0}^\infty\in l^p(\mathbb{N})$. Then $f_k\in (l^p(\mathbb{N}))'$, the dual space of $\ell^p(\mathbb{N})$, and $\|f_k\|=1$.
If $1<p<\infty$, for $f\in (l^p(\mathbb{N}))'$, there is unique $\xi=\{\xi_k\}_{k=0}^\infty\in l^q(\mathbb{N})$ such that $f=\sum\limits_{k=0}^\infty \xi_k f_k$, where $q=\frac{p}{p-1}$. In this case, $\|f\|=(\sum\limits_{k=0}^{\infty}|\xi_k|^q)^{\frac{1}{q}}$.
If $p=1$, for $f\in (l^p(\mathbb{N}))'$, there is unique $\xi=\{\xi_k\}_{k=0}^\infty\in l^{\infty}(\mathbb{N})$ such that $f=\sum\limits_{k=0}^\infty \xi_k f_k$ in weak-star topology, and $\|f\|=\sup\limits_{k}|\xi_k|$.

For convenience, we write
$$\|f\|_q=\begin{cases}\big(\sum\limits_{k=0}^{\infty}|\xi_k|^q\big)^{\frac{1}{q}}&\textrm{ if }
q\neq \infty;\\
\sup\limits_{k}|\xi_k|&\textrm{ if } q=\infty.\end{cases}$$

\begin{lem}\label{kx>ke0}
Let $T$ be a quasinilpotent forward unilateral weighted shift on $\ell^p(\mathbb{N})$
with weight sequence $\{w_n\}_{n=0}^{\infty}$.
If $k_{e_0}\neq 0$,
then
$k_x\geq k_{e_0}$ for any non-zero $x \in l^p(\mathbb{N})$.
\end{lem}
\begin{proof}
For $n\geq0$, we have
\begin{align*}
(z-T)^{-1}e_n&=\sum\limits_{k=0}^{\infty}\frac{1}{z^{k+1}}T^k e_n=\sum\limits_{k=0}^{\infty}
\frac{w_nw_{n+1}\cdots w_{n+k-1}}{z^{k+1}}e_{n+k}\\
&=\sum\limits_{k=0}^{\infty}\frac{\beta_{n+k}}{\beta_n z^{k+1}}e_{n+k}=\beta_n^{-1}z^{n-1}\sum\limits_{k=n}^{\infty}\frac{\beta_k}{z^k}e_k.
\end{align*}
Suppose that $x=\sum\limits_{n=0}^{\infty}\hat{x}_{n}e_n\in l^{p}(\mathbb{N})$ and $x\neq 0$,  then there exists a non-negative integer $n_0$
such that $\hat{x}_{n_0}\neq 0$ and $\hat{x}_{j}=0, \forall j<n_0$. Thus,
\begin{align*}
{(z-T)^{-1}x}&=\sum\limits_{n=0}^{\infty}\hat{x}_{n}(z-T)^{-1}e_n
=\sum\limits_{n=n_0}^{\infty}\hat{x}_{n}(z-T)^{-1}e_n\\
&=\sum\limits_{n=n_0}^{\infty}\hat{x}_{n}
\beta_{n}^{-1}z^{n-1}\sum\limits_{k=n}^{\infty}\frac{\beta_k}{z^k}e_k
=\sum\limits_{k=n_0}^{\infty}\frac{\beta_k}{z^k}(\sum\limits_{n=n_0}^{k}\hat{x}_{n}\beta_{n}^{-1}z^{n-1})e_k.
\end{align*}
It follows that
\begin{align*}
\|(z-T)^{-1}x\|_p
&=\sup\Big\{|f((z-T)^{-1}x)|: f\in(l^p(\mathbb{N}))', \|f\|=1\Big\}\\
&=\sup\Big\{\sum\limits_{k=n_0}^{\infty}\big|\frac{\xi_k\beta_k}{z^k}\big|\big|\sum\limits_{n=n_0}^{k}\hat{x}_n\beta_n^{-1}z^{n-1}\big|:
f=\{\xi_k\}_{k=0}^{\infty}\in l^q(\mathbb{N}), \|f\|_q=1\Big\}\\
&\geq \sum\limits_{k=n_0}^{\infty}\big|\frac{\xi_k\beta_k}{z^k}\big|\big|\sum\limits_{n=n_0}^{k}\hat{x}_n\beta_n^{-1}z^{n-1}\big|.
\end{align*}
We next compute the integral mean value of $\|(z-T)^{-1}x\|_p$.
Let $z=re^{i\theta}$ with $r=|z|$. Then
\begin{align*}
\frac{1}{2\pi}\int_{0}^{2\pi}\|(re^{i\theta}-T)^{-1}x\|_p d\theta
&\geq
\frac{1}{2\pi}\int_{0}^{2\pi}\sum\limits_{k=n_0}^{\infty}\big|\frac{\xi_k\beta_k}{(re^{i\theta})^k}\big|\big|\sum\limits_{n=n_0}^{k}\hat{x}_n\beta_n^{-1}(re^{i\theta})^{n-1}\big|d\theta\\
&=\sum\limits_{k=n_0}^{\infty}\big|\frac{\xi_k\beta_k}{r^k}\big|\frac{1}{2\pi}\int_{0}^{2\pi}\big|\sum\limits_{n=n_0}^{k}\hat{x}_n\beta_n^{-1}r^{n-1}e^{i(n-n_0)\theta}\big|d\theta\\
&\geq\sum\limits_{k=n_0}^{\infty}\big|\frac{\xi_k\beta_k}{r^k}\big|\frac{1}{2\pi}\big|\int_{0}^{2\pi}\sum\limits_{n=n_0}^{k}\hat{x}_n\beta_n^{-1}r^{n-1}e^{i(n-n_0)\theta}d\theta\big|\\
&=\sum\limits_{k=n_0}^{\infty}\big|\frac{\xi_k\beta_k}{r^k}\big|\big|\hat{x}_{n_0}\beta_{n_0}^{-1}r^{n_{0}-1}\big|\\
&\geq\big|\hat{x}_{n_0}\big|\big|\sum\limits_{k=n_0}^{\infty}\frac{\xi_k\beta_k}{r^k}\beta_{n_0}^{-1}r^{n_{0}-1}\big|\\
&=\big|\hat{x}_{n_0}\big|\big|f((z-T)^{-1}e_{{n_0}})\big|,
\end{align*}
where $f=\{\xi_k\}_{k=0}^{\infty}\in (l^p(\mathbb{N}))'$ and $\|f\|_q=1$. Thus for any $|z|=r$, we have
\begin{align*}
\sup_{\theta\in[0,2\pi]}\|(re^{i\theta}-T)^{-1}x\|_p
&\geq{\frac{1}{2\pi}\int_{0}^{2\pi}\|(re^{i\theta}-T)^{-1}x\|_p d\theta}\\
&\geq |x_{n_0}|\|(z-T)^{-1}e_{n_0}\|_p.
\end{align*}
Pick $z=re^{i\theta}$ so that
$$\|(z-T)^{-1}x\|_p=\sup_{\theta\in[0,2\pi]}\|(re^{i\theta}-T)^{-1}x\|_p.$$
Thereby,
\begin{equation}\label{F:re}
\|(re^{i\theta}-T)^{-1}x\|_p\geq|\hat{x}_{n_0}|\|(re^{i\theta}-T)^{-1}e_{n_0}\|_p.
\end{equation}
Note that for $z\neq 0$,
$$\|(z-T)^{-1}e_{n_0}\|_p=\|(|z|-T)^{-1}e_{n_0}\|_p.$$
And since $T$ have circular symmetry, we have that
$$\|(z-T)^{-1}\|=\|(|z|-T)^{-1}\|, ~~\forall z\neq 0.$$
So
\begin{align*}
k_{e_{n_0}}=\limsup\limits_{|z|\rightarrow 0}\frac{\ln\|(|z|-T)^{-1}e_{n_0}\|_p}{\ln\|(|z|-T)^{-1}\|}.
\end{align*}
We can choose a sequence $\{r_j\}_{j=0}^{\infty}\subset[0,+\infty)$ with $\lim\limits_{j\rightarrow \infty}r_j=0$ such that
\begin{equation}\label{F:ken0}
k_{e_{n_0}}=\lim\limits_{j\rightarrow \infty}\frac{\ln\|(r_j-T)^{-1}e_{n_0}\|_p}{\ln\|(r_j-T)^{-1}\|}.
\end{equation}
By (\ref{F:re}), one can pick a sequence $\{z_j\}_{j=0}^{\infty}\subset\mathbb{C}$ with $|z_j|=r_j$ such that
\begin{equation}\label{F:ken}
\liminf_{j\rightarrow \infty}\frac{\ln\|(z_j-T)^{-1}x\|_p}{\ln\|(z_j-T)^{-1}e_{n_0}\|_p}\geq 1.
\end{equation}
Combining (\ref{F:ken0}) and (\ref{F:ken}), we can deduce that
\begin{align*}
{k_x}&=\limsup_{z\rightarrow 0}\frac{\ln\|(z-T)^{-1}x\|_p}{\ln\|(z-T)^{-1}\|}\\
&\geq\liminf_{j\rightarrow \infty}\frac{\ln\|(z_j-T)^{-1}x\|_p}{\ln\|(z_j-T)^{-1}\|}\\
&=\liminf_{j\rightarrow \infty}\frac{\ln\|(z_j-T)^{-1}x\|_p}{\ln\|(z_j-T)^{-1}e_{n_0}\|_p}
\lim_{j\rightarrow \infty}\frac{\ln\|(z_j-T)^{-1}e_{n_0}\|_p}{\ln\|(z_j-T)^{-1}\|}\\
&=\liminf_{j\rightarrow \infty}\frac{\ln\|(z_j-T)^{-1}x\|_p}{\ln\|(z_j-T)^{-1}e_{n_0}\|_p}
\lim_{j\rightarrow \infty}\frac{\ln\|(r_j-T)^{-1}e_{n_0}\|_p}{\ln\|(r_j-T)^{-1}\|}\\
&\geq k_{e_{n_0}}.
\end{align*}
Moreover, since $T^{n_{0}} e_0=\prod\limits_{j=0}^{n_{0}-1}w_j e_{n_0}$ and by Lemma \ref{F:Tx=x}, 
we have $k_{e_{n_0}}=k_{e_{0}}$.
Thus, it holds that $k_x\geq k_{e_0}$ for any non-zero $x\in l^p(\mathbb{N})$.
\end{proof}

From the above lemma, we can obtain the following theorem.
\begin{thm}\label{F:ke01}
Suppose that $T$ is a quasinilpotent forward unilateral weighted shift on $\ell^p(\mathbb{N})$.
If $k_{e_0}=1$, then $\Lambda(T)=\{1\}$.
\end{thm}
\begin{proof}
For any non-zero $x\in l^p(\mathbb{N})$, by Lemma $\ref{kx>ke0}$ and the fact $0\leq k_x\leq 1$, we have
$k_x=1$.
Hence, $\Lambda(T)=\{1\}$.
\end{proof}
Next, we apply Theorem \ref{F:ke01} to give the following two corollaries.

\begin{cor}
Let $T$ is a forward unilateral weighted shift on $l^1(\mathbb{N})$ with weight sequence $\{w_n\}_{n=0}^\infty$. If $\{w_n\}_{n=0}^\infty$ is decreasing to $0$, then $\Lambda(T)=\{1\}$.
\end{cor}

\begin{proof}
It is easy to check that
$$\|T^{n}\|=\|T^{n}e_0\|=\prod\limits_{j=0}^{n-1}w_j$$ and $\sigma(T)=\{0\}$.
For $z\neq 0$, we have
\begin{align*}
\|(z-T)^{-1}\|
=\big\|\sum\limits_{n=0}^{\infty}\frac{T^n}{z^{n+1}}\big\|
&\leq\sum\limits_{n=0}^{\infty}\frac{\|T^n\|}{|z|^{n+1}}
=\sum\limits_{n=0}^{\infty}\frac{1}{|z|^{n+1}}\prod\limits_{j=0}^{n-1}w_j
\end{align*}
and 
$\|(z-T)^{-1}e_{0}\|_1\leq \|(z-T)^{-1}\|$.
Thus, $\|(z-T)^{-1}e_{0}\|_1=\|(z-T)^{-1}\|$
and
\begin{equation*}
k_{e_0}=\limsup_{z\rightarrow0}\frac{\ln\|(z-T)^{-1}e_{0}\|_1}{\ln\|(z-T)^{-1}\|}=\limsup_{z\rightarrow0}\frac{\ln\|(z-T)^{-1}\|}{\ln\|(z-T)^{-1}\|}=1.
\end{equation*}
By Theorem \ref{F:ke01}, we have $\Lambda{(T)}=\{1\}$.
\end{proof}
\begin{rem}
In \cite{ME73}, Embry showed that if $T$ is a forward unilateral weighted shift on $l^1(\mathbb{N})$ with weight sequence decreasing to $0$, then $T$ is strictly cyclic.
However, this is not valid for $p>1$.
For counterexample, we can see \cite{GHF78}.
Next, we will prove that if $T$ is strictly cyclic on $\ell^p(\mathbb{N})$ for $1\leq p<\infty$, then $\Lambda(T)=\{1\}$.
\end{rem}

Recall that $T\in \mathcal{B}(X)$ is called strictly cyclic if there exists a vector $x_{0}\in X$ such that $\{Ax_{0}:A \in \mathcal{A}(T)\}=X$,
where $\mathcal{A}(T)$ denote the closed subalgebra generated by the identity $I$ and $T$ in the weak operator topology.
Such vector $x_{0}$ is called a strictly cyclic vector for $T$ (see \cite{Shs}).
\begin{cor}\label{F:sc1}
Let $T$ be a strictly cyclic quasinilpotent forward unilateral weighted shift on $\ell^p(\mathbb{N})$ for $1\leq p<\infty$, then $\Lambda(T)=\{1\}$.
\end{cor}

\begin{proof}
 Obviously, $e_0$ is a cyclic vector for $T$. Let $\Phi$ be the map from $\mathcal{A}(T)$ to $X$ such that
$\Phi(A)=Ae_0$ for $A\in \mathcal{A}(T)$.
Then $\Phi$ is invertible (see {\cite[Page 93]{Shs}}).
So, there exists a constant $c>0$ such that
$$c\|A\|\leq\|Ae_0\|\leq\|A\|,\quad\forall A\in\mathcal{A}(T).$$
Since $T$ quasinilpotent, for $z\neq 0$, $(z-T)^{-1}\in \mathcal{A}(T)$.
Thus
\begin{align*}
\frac{\ln\|(z-T)^{-1}e_0\|_p}{\ln \|(z-T)^{-1}\|}
\geq\frac{\ln C\|(z-T)^{-1}\|}{\ln\|(z-T)^{-1}\|}
=\frac{\ln C}{\ln\|(z-T)^{-1}\|}+1
\end{align*}
and $\liminf\limits_{z\rightarrow 0}\frac{\ln\|(z-T)^{-1}e_0\|_p}{\ln\|(z-T)^{-1}\|}\geq 1$. It follows from $0\leq k_{e_0}\leq 1$ that
\begin{align*}
k_{e_0}=\lim\limits_{z\rightarrow 0}\frac{\ln\|(z-T)^{-1}e_0\|_p}{\ln\|(z-T)^{-1}\|}=1.
\end{align*}
By Theorem \ref{F:ke01}, we have $\Lambda(T)=\{1\}$.
\end{proof}


\section{Some backward weighted shifts with power set $[0,1]$}

For $1\leq p<\infty$, let $T$ be a quasinilpotent backward unilateral weighted shift on $\ell^p(\mathbb{N})$ with weight sequence $\{w_n\}_{n=1}^{\infty}$.
Given $m\geq1$, for $z\neq0$,
\begin{align*}
\|(z-T)^{-1}e_m\|_p^p
&=\|\sum\limits_{i=0}^{m}\frac{1}{z^{i+1}}T^ie_m\|_p^p
=\|\frac{e_m}{z}+\sum\limits_{i=1}^{m}\big(\frac{1}{z^{i+1}}\prod\limits_{j=0}^{i-1}w_{m-j}\big)e_{m-i}\|_p^p\\
&=\frac{1}{|z|^p}+\sum\limits_{i=1}^{m}\Big|\frac{1}{z^{i+1}}\prod\limits_{j=0}^{i-1}w_{m-j}\Big|^p.
\end{align*}
Since $\|(z-T)^{-1}e_N\|_p\leq\|(z-T)^{-1}\|$ for any $N \in\mathbb{N}$, it follows that
\begin{align*}
k_{e_m}&=\limsup_{z\rightarrow0}\frac{\ln\|(z-T)^{-1}e_m\|_p}{\ln\|(z-T)^{-1}\|}\\
&\leq\limsup_{z\rightarrow0}\frac{\ln\|(z-T)^{-1}e_m\|_p}{\ln\|(z-T)^{-1}e_N\|_p}
=\frac{m+1}{N+1}.
\end{align*}
Thus, we have $k_{e_m}=0$, because the above formula holds for any $N \in\mathbb{N}$.
Further, it is obvious that $k_{e_0}=0$.
Therefore, we conclude that $k_{e_m}=0$ for every $ m\geq0$.
From this, it is easy to check that $k_x=0$ when $x\in \operatorname{span}\{e_n: n\geq 0\}$.
So, this gives people a false impression that $\Lambda(T)=\{0\}$.
Actually, that is not the case.
In this section, we will prove that the power set of a class of backward unilateral weighted shifts is $[0,1]$.

Throughout this section, for $k\geq 0$, we set $f_k(x)=x_k$ for any $x=\{x_n\}_{n=0}^\infty\in l^p(\mathbb{N})$.
We first give the following Theorem.

\begin{thm}\label{B:contain1}
Let $T$ be a backward unilateral weighted shift on $\ell^p(\mathbb{N})$
with decreasing weight sequence $\{w_n\}_{n=1}^{\infty}$.
If $\{w_n\}_{n=1}^{\infty}$ is $p^{\prime}$-summable for some positive number $p^{\prime}$,
then $1\in\Lambda (T)$.
\end{thm}

\begin{proof}
Obviously, $T$ is a quasinilpotent operator. Since $\{w_n\}_{n=1}^{\infty}$ is $p^{\prime}$-summable for some positive number $p^{\prime}$,
one can pick $m\geq 1$ so that the sequence $\{\prod\limits_{j=1}^{m}w_{n+j}\}_{n=0}^{\infty}$ is summable.
Set $x=\{\alpha_n\}_{n=0}^{\infty}$, where $\alpha_n=\prod\limits_{j=1}^{m}w_{n+j}$.
Then $x\in l^p(\mathbb{N})$ for any $p\in [1,+\infty)$.
We can first observe that
$$\|(z-T)^{-1}x\|_p\geq \big|f_0\big((z-T)^{-1}x\big)\big|,\quad \forall z\neq 0,$$
and
\begin{align*}
k_{x}
=\limsup_{z\rightarrow 0}\frac{\ln\|(z-T)^{-1}x\|_p}{\ln\|(z-T)^{-1}\|}
\geq\limsup_{z\rightarrow 0}\frac{\ln\big|f_0\big((z-T)^{-1}x\big)\big|}{\ln\|(z-T)^{-1}\|}.
\end{align*}
Next, we compute the value of $f_0\big((z-T)^{-1}x\big)$.
For $z\neq 0$,
\begin{equation}\label{B:f0}
\begin{aligned}
f_0\big((z-T)^{-1}x\big)
&=\sum\limits_{n=0}^{\infty}\frac{f_0(T^n x)}{z^{n+1}}
=\frac{\alpha_0}{z}+\sum\limits_{n=1}^{\infty}\frac{\prod\limits_{j=1}^{n}w_j}{z^{n+1}}\alpha_n\\
&=\frac{\alpha_0}{z}+\sum\limits_{n=1}^{\infty}\frac{\prod\limits_{j=1}^{n+m}w_{j}}{z^{n+1}}
=\frac{\alpha_0}{z}+z^m\sum\limits_{n=1}^{\infty}\frac{\prod\limits_{j=1}^{n+m}w_{j}}{z^{n+m+1}}\\
&=\frac{\alpha_0}{z}+z^m\sum\limits_{k=m+1}^{\infty}\frac{\prod\limits_{j=1}^{k}w_{j}}{z^{k+1}}=\frac{\alpha_0}{z}
+z^m\big(\sum\limits_{k=1}^{\infty}\frac{\prod\limits_{j=1}^{k}w_{j}}{z^{k+1}}
-\sum\limits_{k=1}^{m}\frac{\prod\limits_{j=1}^{k}w_{j}}{z^{k+1}}\big).
\end{aligned}
\end{equation}
Furthermore, since $\{w_n\}_{n=1}^{\infty}$ is decreasing, it follows that $\|T^k\|=\prod\limits_{j=1}^{k}w_{j}$ and
\begin{equation}\label{B:normT}
\|(z-T)^{-1}\|\leq\frac{1}{|z|}+\sum\limits_{k=1}^{\infty}\frac{\|T^k\|}{|z|^{k+1}}
=\frac{1}{|z|}+\sum\limits_{k=1}^{\infty}\frac{\prod\limits_{j=1}^{k}w_{j}}{|z|^{k+1}}.
\end{equation}
Therefore, combining (\ref{B:f0}) and (\ref{B:normT}), we can deduce that
\begin{equation}\label{B:kx=1}
\begin{aligned}
k_{x}
&\geq
\limsup_{z\rightarrow 0}\frac{\ln |f_0\big((z-T)^{-1}x\big)|}{\ln\|(z-T)^{-1}\|}\\
&\geq\limsup_{z\rightarrow 0}\frac{\ln \big|\frac{\alpha_0}{z}
+z^m\big(\sum\limits_{k=1}^{\infty}\frac{\prod\limits_{j=1}^{k}w_{j}}{z^{k+1}}
-\sum\limits_{k=1}^{m}\frac{\prod\limits_{j=1}^{k}w_{j}}{z^{k+1}}\big)\big|}
{\ln\Big(\frac{1}{|z|}+\sum\limits_{k=1}^{\infty}\frac{\prod\limits_{j=1}^{k}w_{j}}{|z|^{k+1}}\Big)}\\
&\geq\limsup_{|z|\rightarrow 0}\frac{\ln\Big(\frac{\alpha_0}{|z|}
+|z|^m\big(\sum\limits_{k=1}^{\infty}\frac{\prod\limits_{j=1}^{k}w_{j}}{|z|^{k+1}}
-\sum\limits_{k=1}^{m}\frac{\prod\limits_{j=1}^{k}w_{j}}{|z|^{k+1}}\big)\Big)}
{\ln\Big(\frac{1}{|z|}+\sum\limits_{k=1}^{\infty}\frac{\prod\limits_{j=1}^{k}w_{j}}{|z|^{k+1}}\Big)}
=1.
\end{aligned}
\end{equation}
By the fact $0 \leq k_{x}\leq1$, we have $k_{x}=1$. Hence $1\in \Lambda(T)$.
\end{proof}

\begin{rem}
By Theorem \ref{B:contain1}, one can see that the power set of a backward weighted shift can also contain $1$.
Therefore, it will be very interesting to investigate its power set.
Next, we will show that the power set of a class of backward weighted shifts is $[0,1]$.
\end{rem}

\begin{lem}\label{B:0to1}
Let $T$ be a quasinilpotent backward unilateral weighted shift on $\ell^p(\mathbb{N})$ with weight sequence $\{w_n\}_{n=1}^\infty$.
If there is a natural number $m$ such that
$x=\{\alpha_n\}_{n=0}^{\infty}$ in $\ell^p(\mathbb{N})$, where $\alpha_0=1$ and $\alpha_n=\big((n+m)!\prod\limits_{j=1}^{n}w_j\big)^{-1}$ for $n\geq1$,
then $[0,k_x]\subset\Lambda(T)$.
\end{lem}

\begin{proof}
For $0<r\leq 1$, let $x_r=\big\{a_n\big\}_{n=0}^{\infty}$, where $a_0=1$ and $a_n=r^n\big((n+m)!\prod\limits_{j=1}^{n}w_j\big)^{-1}$ for $n\geq1$.
Then $x_1=x$ and $x_r\in l^p(\mathbb{N})$.
For $z\neq 0$,
\begin{equation}\label{fxr}
\begin{aligned}
f_0\big((|z|-T)^{-1}x_r\big)
&=\sum\limits_{k=0}^{\infty}\frac{f_0(T^k x_r)}{|z|^{k+1}}
=\frac{1}{|z|}\Big(1+\sum\limits_{k=1}^{\infty}\frac{1}{(k+m)!}\frac{r^k}{|z|^k}\Big)\\
&=\frac{1}{|z|}\Big(1+\frac{|z|^{m}}{r^{m}}\sum\limits_{k=1}^{\infty}\frac{1}{(k+m)!}\frac{r^{k+m}}{|z|^{k+m}}\Big)\\
&=\frac{1}{|z|}+\frac{|z|^{m-1}}{r^{m}}\sum\limits_{n=m+1}^{\infty}\frac{1}{n!}\frac{r^n}{|z|^n}\\
&=\frac{1}{|z|}+\frac{|z|^{m-1}}{r^{m}}\Big(\sum\limits_{n=1}^{\infty}\frac{1}{n!}\frac{r^n}{|z|^n}
-\sum\limits_{n=1}^{m}\frac{1}{n!}\frac{r^n}{|z|^n}\Big)\\
&=\frac{1}{|z|}+\frac{|z|^{m-1}}{r^{m}}\Big(e^{\frac{r}{|z|}}-\sum\limits_{n=0}^{m}\frac{1}{n!}\frac{r^n}{|z|^n}\Big).
\end{aligned}
\end{equation}
It follows that
\begin{equation}\label{fr}
\begin{aligned}
\lim_{z\rightarrow 0}\frac{\ln\big|f_0\big((|z|-T)^{-1}x_r\big)\big|}{\ln\big|f_0\big((|z|-T)^{-1}x\big)\big|}=\lim_{|z|\rightarrow 0}\frac{\ln\Big(\frac{1}{|z|}+\frac{|z|^{m-1}}{r^{m}}\big(e^{\frac{r}{|z|}}-\sum\limits_{n=0}^{m}\frac{1}{n!}\frac{r^n}{|z|^n}\big)\Big)}
{\ln\Big(\frac{1}{|z|}+|z|^{m-1}\big(e^{\frac{1}{|z|}}-\sum\limits_{n=0}^{m}\frac{1}{n!}\frac{1}{|z|^n}\big)\Big)}=r.
\end{aligned}
\end{equation}
Moreover, for any $ n\geq 1$ we have
\begin{align*}
f_n\big((z-T&)^{-1}x_r\big)
=\sum\limits_{k=0}^{\infty}\frac{f_n(T^k x_r)}{z^{k+1}}\\
&=\frac{1}{(n+m)!\prod\limits_{j=1}^n w_j}\frac{r^n}{z}+\frac{w_{n+1}}{(n+m+1)!\prod\limits_{j=1}^{n+1} w_j}\frac{r^{n+1}}{z^2}
+\frac{w_{n+1}w_{n+2}}{(n+m+2)!\prod\limits_{j=1}^{n+2} w_j}\frac{r^{n+2}}{z^3}+\cdots\\
&=\frac{1}{(n+m)!\prod\limits_{j=1}^n w_j}\frac{r^n}{z}\Big(1+\frac{1}{n+m+1}\frac{r}{z}+\frac{1}{(n+m+1)(n+m+2)}\frac{r^2}{z^2}+\cdots\Big)
\end{align*}
and
\begin{align*}
|f_n\big((z-T)^{-1}x_r\big)|
&\leq \frac{1}{(n+m)!\prod\limits_{j=1}^n w_j}\frac{r^n}{|z|}\Big(1+\frac{1}{n+m+1}\frac{r}{|z|}+\frac{1}{(n+m+1)(n+m+2)}\frac{r^2}{|z|^2}+\cdots\Big)\\
&\leq\frac{1}{(n+m)!\prod\limits_{j=1}^n w_j}\frac{r^n}{|z|}\Big(1+\frac{1}{(m+1)}\frac{r}{|z|}+\frac{1}{(m+1)(m+2)}\frac{r^2}{|z|^2}+\cdots\Big)\\
&=\frac{r^n}{(n+m)!\prod\limits_{j=1}^n w_j}|f_0\big((|z|-T)^{-1}x_r\big)|.
\end{align*}
It follows that
\begin{equation}\label{xrn}
\begin{aligned}
\|(z-T)^{-1}x_r\|_p
&=\Big(\sum\limits_{n=0}^{\infty}\big|f_n\big((z-T)^{-1}x_r\big)\big|^p\Big)^\frac{1}{p}\\
&\leq \big|f_0\big((|z|-T)^{-1}x_r\big)\big|\bigg(1+\sum\limits_{n=1}^{\infty}\Big(\frac{r^n}{(n+m)!\prod\limits_{j=1}^n w_j}\Big)^p\bigg)^\frac{1}{p}\\
&=\big|f_0\big((|z|-T)^{-1}x_r\big)\big|\|x_r\|_p.
\end{aligned}
\end{equation}
Moreover, it is obvious that $\|(z-T)^{-1}x_r\|_p\geq \big|f_0\big((z-T)^{-1}x_r\big)\big|$, $\forall z\neq 0$. So
\begin{align*}
\frac{\ln\big|f_0\big((|z|-T)^{-1}x_r\big)\big|}{\ln\|(z-T)^{-1}\|}
&\leq\frac{\ln\|(z-T)^{-1}x_r\|_p}{\ln\|(z-T)^{-1}\|}
\leq \frac{\ln\big|f_0\big((|z|-T)^{-1}x_r\big)\big|\|x_r\|_p}{\ln\|(z-T)^{-1}\|}.
\end{align*}
On the other hand, since $T$ have circular symmetry, we have
$$\|(z-T)^{-1}\|=\|(|z|-T)^{-1}\|,~~ \forall z\neq 0.$$
Then, we can conclude that
\begin{align}\label{r01}
k_{x_r}=\limsup\limits_{|z|\rightarrow 0}\frac{\ln\big|f_0\big((|z|-T)^{-1}x_r\big)\big|}{\ln\|(|z|-T)^{-1}\|},~~ \forall r\in (0,1].
\end{align}
Combining $(\ref{fr})$ and $(\ref{r01})$, we can deduce that
\begin{align*}
k_{x_r}
&=\limsup_{z\rightarrow 0}\frac{\ln\|(z-T)^{-1}x_r\|_p}{\ln\|(z-T)^{-1}\|}\\
&=\lim_{|z|\rightarrow0}\frac{\ln\big|f_0\big((|z|-T)^{-1}x_r\big)\big|}{\ln\big|f_0\big((|z|-T)^{-1}x\big)\big|}
\limsup_{|z|\rightarrow0}\frac{\ln\big|f_0\big((|z|-T)^{-1}x\big)\big|}{\ln\|(|z|-T)^{-1}\|}\\
&=r\cdot\limsup_{|z|\rightarrow 0}\frac{\ln\big|f_0\big((|z|-T)^{-1}x\big)\big|}{\ln\|(|z|-T)^{-1}\|}=rk_x.
\end{align*}
Since $0< r\leq1$ and $0\in \Lambda(T)$,  it holds that $[0,k_x]\subset\Lambda(T)$.
\end{proof}

\begin{thm}\label{B:pw0t1}
Let $T$ be a backward unilateral weighted shift on $\ell^p(\mathbb{N})$ with weight sequence $\{w_n\}_{n=1}^\infty$.
If there is a positive number $m_0$ such that $\frac{1}{n+2m_0}\leq w_{n+m_0}\leq\frac{1}{n}$ for $n\geq1$, then $\Lambda(T)=[0,1]$.
 \end{thm}

\begin{proof}
Pick a positive number $m>2m_0+2$, then for $n>m_0$ we have
\begin{align*}
\big((n+m)!\prod\limits_{j=1}^{n}w_j\big)^{-1}
&=\big((n+m)!\prod\limits_{j=1}^{m_0}w_j\prod\limits_{j=m_0+1}^{n}w_j\big)^{-1}\\
&\leq(\prod\limits_{j=1}^{m_0}w_j)^{-1}\frac{(n+2m_0)!}{(n+m)!}
\leq(\prod\limits_{j=1}^{m_0}w_j)^{-1}\frac{1}{(n+m)(n+m+2)}.
\end{align*}
From this, we have
\begin{equation}\label{pli}
\begin{aligned}
\sum\limits_{n=1}^{\infty}\big((n+m)!\prod\limits_{j=1}^{n}w_j\big)^{-1}
\leq\sum\limits_{n=1}^{m_0}\big((n+m)!\prod\limits_{j=1}^{n}w_j\big)^{-1}
+\sum\limits_{n=m_0+1}^{\infty}\frac{(\prod\limits_{j=1}^{m_0}w_j)^{-1}}{(n+m)(n+m+2)}<\infty.
\end{aligned}
\end{equation}
Let $x=\{\alpha_n\}_{n=0}^{\infty}$, where $\alpha_0=1$ and $\alpha_n=\big((n+m)!\prod\limits_{j=1}^{n}w_j\big)^{-1}$ for $n\geq1$.
Then $x\in l^p(\mathbb{N})$ for any $p\in[1,+\infty)$.
By Lemma $\ref{B:0to1}$, have $[0,k_x]\subset\Lambda(T)$.
Next, we prove $k_x=1$.
From $w_{n+m_0}\leq\frac{1}{n}$, we can deduce that
\begin{align*}
\|(z-T)^{-1}\|
&\leq\sum\limits_{k=0}^{\infty}\frac{\|T^k\|}{|z|^{k+1}}
=\frac{1}{|z|}+\sum\limits_{k=1}^{\infty}\frac{\prod\limits_{j=1}^k w_j}{|z|^{k+1}}\\
&=\frac{1}{|z|}+\sum\limits_{k=1}^{m_0}\frac{\prod\limits_{j=1}^k w_j}{|z|^{k+1}}+\sum\limits_{k=m_0+1}^{\infty}\frac{\prod\limits_{j=1}^k w_j}{|z|^{k+1}}\\
&\leq\frac{1}{|z|}+\sum\limits_{k=1}^{m_0}\frac{\prod\limits_{j=1}^k w_j}{|z|^{k+1}}+\sum\limits_{k=m_0+1}^{\infty}\frac{1}{k!|z|^{k+1}}\\
&=\frac{1}{|z|}+\sum\limits_{k=1}^{m_0}\frac{\prod\limits_{j=1}^k w_j}{|z|^{k+1}}+\frac{1}{|z|}\big(\sum\limits_{k=1}^{\infty}\frac{1}{k!|z|^{k}}
-\sum\limits_{k=1}^{m_0}\frac{1}{k!|z|^{k}}\big)\\
&=\frac{1}{|z|}+\sum\limits_{k=1}^{m_0}\frac{\prod\limits_{j=1}^k w_j}{|z|^{k+1}}+\frac{1}{|z|}\big(e^{\frac{1}{|z|}}-\sum\limits_{k=0}^{m_0}\frac{1}{k!|z|^{k}}\big).
\end{align*}
Combining (\ref{fxr}), we can obtain that
\begin{align*}
k_{x}
&=\limsup_{z\rightarrow 0}\frac{\ln\|(z-T)^{-1}x\|_p}{\ln\|(z-T)^{-1}\|}\\
&\geq\limsup_{z\rightarrow 0}\frac{\ln |f_0\big((z-T)^{-1}x\big)|}{\ln\|(z-T)^{-1}\|}\\
&\geq\limsup_{|z|\rightarrow 0}
\frac{\ln\Big(\frac{1}{|z|}+|z|^{m-1}\big(e^{\frac{1}{|z|}}-\sum\limits_{n=0}^{m}\frac{1}{n!}\frac{1}{|z|^n}\big)\Big)}
{\ln \Big(\frac{1}{|z|}+\sum\limits_{k=1}^{m_0}\frac{\prod\limits_{j=1}^k w_j}{|z|^{k+1}}+\frac{1}{|z|}\big(e^{\frac{1}{|z|}}-\sum\limits_{k=0}^{m_0}\frac{1}{k!|z|^{k}}\big)\Big)}=1.
\end{align*}
Since $0\leq k_x \leq 1$, it yields that $k_x=1$. Hence $\Lambda(T)=[0,1]$.
\end{proof}

\begin{cor}\label{B:buwn1}
Let $T$ be the backward unilateral weighted shift with weight sequence $\{\frac{1}{n+1}\}_{n=0}^\infty$ on $\ell^p(\mathbb{N})$ for $1\leq p<\infty$. Then $\Lambda(T)=[0,1]$.
\end{cor}
\begin{proof}
It follows immediately from Theorem $\ref{B:pw0t1}$.
\end{proof}


\section{A bilateral weighted shift with power set $[\frac{1}{2},1]$}

For $1<p<\infty$, let $A$ be an injective forward unilateral weighted shift on $\ell^p(\mathbb{N})$ with weight sequence $\{w_n\}_{n=1}^{\infty}$.
If $\{w_n\}_{n=1}^{\infty}$ is monotone decreasing and $p^{\prime}$-summable for some positive number $p^{\prime}$, then $A$ is strictly cyclic (see {\cite[Corollary 3.5]{KL73}}).
So, by Corollary $\ref{F:sc1}$, we have $\Lambda(A)=\{1\}$.
From this, the aim of this section is to construct a bilateral weighted shift with power set $[\frac{1}{2},1]$ on $\ell^p(\mathbb{Z})$ for $1<p<\infty$.

For $k\in \mathbb{Z}$, set $f_k(x)=x_k$, $\forall x=\{x_n\}_{n=-\infty}^{+\infty}\in l^p(\mathbb{Z})$.
Let $e_1\otimes f_{0}$ denote the rank-one operator on $\ell^p(\mathbb{Z})$ defined as $(e_1\otimes f_{0})(x)=f_{0}(x)e_1, \forall x\in l^p(\mathbb{Z})$.
Now, we defined a bilateral weighted shift on $\ell^p(\mathbb{Z})$ as follows:
\begin{equation}\label{bmT}
T=\begin{bmatrix}
B&0\\
e_1\otimes f_{0}&A
\end{bmatrix}
\end{equation}
where
\[
A=
\begin{bmatrix}
  0  \\
  w_1&0\\
  &w_2&0\\
  &&w_3&0\\
  &&&\ddots&\ddots
\end{bmatrix}
\begin{matrix}
e_1\\
e_2\\
e_3\\
e_4\\
\vdots
\end{matrix}, \ \ \ \
B=
\begin{bmatrix}
  0&w_1 \\
  &0&w_2\\
  &&0&w_3\\
  &&&0&w_4\\
  &&&&\ddots&\ddots
\end{bmatrix}
\begin{matrix}
e_0\\
e_{-1}\\
e_{-2}\\
e_{-3}\\
\vdots
\end{matrix}
.\]

Note that $A=B^{tr}$, where $B^{tr}$ is the transpose of $B$. And it is not difficult to prove that $\|(z-A)^{-1}\|=\|(z-B)^{-1}\|$ for $z\neq 0$.
Then we have the following lemmas.

\begin{lem}\label{B:AT}
Let $T$ be as in $(\ref{bmT})$ above. If $\{w_n\}_{n=1}^{\infty}$ is monotone decreasing and $p^{\prime}$-summable for some positive number $p^{\prime}$, then $\Lambda(T)\subset[\frac{1}{2},1]$.
\end{lem}

\begin{proof}
We first compute the value of $\|(z-T)^{-1}\|$.
For $z\neq 0$,
\begin{equation}\label{norm(z-T)}
(z-T)^{-1}=
\begin{bmatrix}
  (z-B)^{-1}&0\\
  -(z-A)^{-1}e_1\otimes(z-B')^{-1}f_0&(z-A)^{-1}
\end{bmatrix},
\end{equation}
where $B'$ is the Banach conjugate of $B$.
It follows that
\begin{equation*}
\|(z-T)^{-1}\|
\leq \|(z-A)^{-1}\|+\|(z-B)^{-1}\|+\|(z-A)^{-1}e_1\|_p\|(z-B')^{-1}f_0\|_q
\end{equation*}
and
\begin{equation*}
\|(z-A)^{-1}e_1\|_p\|(z-B')^{-1}f_0\|_q\leq\|(z-T)^{-1}\|,
\end{equation*}
where $q=\frac{p}{p-1}$.
Noting that $e_1$ is a strictly cyclic vector of $A$. So, there exists a constant $c>0$ such that
$$c\|(z-A)^{-1}\|\leq\|(z-A)^{-1}e_1\|_p\leq \|(z-A)^{-1}\|.$$
From this, we have
\begin{equation}\label{T:Ae1}
k_{(A,e_1)}=\lim\limits_{z\rightarrow 0}\frac{\ln\|(z-A)^{-1}e_1\|_p}{\ln\|(z-A)^{-1}\|}=1.
\end{equation}
For convenience, given two functions $X(z), Y(z): \mathbb{C}\backslash{\{0\}}\rightarrow [0,\infty)$, we write $X(z)\approx Y(z)$ if $\lim\limits_{z\rightarrow 0}\frac{X(z)}{Y(z)}=1$.
Hence, when $z$ tends to $0$, we have that
$$\ln\|(z-A)^{-1}\|\approx\ln\|(z-A)^{-1}e_1\|_p.$$
Similarly, we can also obtain
$$\ln\|(z-A)^{-1}\|=\ln\|(z-B)^{-1}\|=\ln\|(z-B')^{-1}\|\approx\ln\|(z-B')^{-1}f_0\|_q.$$
Therefore, as $z$ tends to $0$, we have
\begin{equation}\label{ZA2}
\begin{aligned}
\ln\|(z-T)^{-1}\|&\approx\ln\|(z-A)^{-1}e_1\|_p\|(z-B')^{-1}f_0\|_q\\
&\approx\ln\|(z-A)^{-1}\|^2.
\end{aligned}
\end{equation}

Next, we prove that $k_{(T,x)}\geq \frac{1}{2}$ holds for any non-zero $x\in l^p(\mathbb{Z})$.
Let
$x=\left[\begin{array}{cc}
  x_1\\
 x_2
\end{array}\right]\in l^p(\mathbb{Z})$ and $x\neq0$. Then
\begin{equation}\label{zTx}
(z-T)^{-1}x=
\begin{bmatrix}
  (z-B)^{-1}x_1\\
  -\big((z-B')^{-1}f_0\big)(x_1)(z-A)^{-1}e_1+(z-A)^{-1}x_2
\end{bmatrix}.
\end{equation}
The rest proof is divided into three cases.

{\bf Case 1.}  $x_1=0, x_2 \neq 0$.

It is clear that $\|(z-T)^{-1}x\|_p=\|(z-A)^{-1}x_2\|_p$.
Since $k_{(A,x_2)}=1$, it follows that
\begin{align*}
k_{(T,x)}
=\limsup_{z\rightarrow 0}\frac{\ln\|(z-T)^{-1}x\|_p}{\ln\|(z-T)^{-1}\|}
&=\limsup_{z\rightarrow 0}\frac{\ln\|(z-A)^{-1}x_2\|_p}{\ln\|(z-T)^{-1}\|}\\
&\approx\limsup_{z\rightarrow 0}\frac{\ln\|(z-A)^{-1}x_2\|_p}{\ln\|(z-A)^{-1}\|^2}=\frac{1}{2}.
\end{align*}

{\bf Case 2.}  $x_1\neq0, x_2 =0$.

In this case, we can observe that $$\|(z-T)^{-1}x\|_p\geq\big|\big((z-B')^{-1}f_0\big)(x_1)\big|\|(z-A)^{-1}e_1\|_p.$$
Notice that
\begin{align*}
\big((z-B')^{-1}f_0\big)(x_1)
=f_0\big((z-B)^{-1}x_1\big)
=\sum\limits_{k=0}^{\infty}\frac{f_0(B^kx_1)}{z^{k+1}}.
\end{align*}
Since $x_1\neq0$, we have that
\begin{align*}
\limsup\limits_{z\rightarrow 0}\big|\big((z-B')^{-1}f_0\big)(x_1)\big|=+\infty.
\end{align*}
We can pick a sequence $\{z_j\}_{j=0}^{\infty}\subset\mathbb{C}$ with $\lim\limits_{j\rightarrow\infty}|z_j|=0$ such that
\begin{equation}\label{BT:geq1}
  |\big((z_j-B')^{-1}f_0\big)(x_1)|>1,\quad \forall j\geq 0.
  \end{equation}
  Combining (\ref{T:Ae1}) and (\ref{BT:geq1}), we can deduce that
  \begin{align*}
  k_{(T,x)}
   &=\limsup_{z\rightarrow 0}\frac{\ln\|(z-T)^{-1}x\|_p}{\ln\|(z-T)^{-1}\|}\\
   &\approx\limsup_{z\rightarrow 0}\frac{\ln\|(z-T)^{-1}x\|_p}{\ln\|(z-A)^{-1}\|^2}\\
   &\geq\limsup_{z\rightarrow 0}\frac{\ln\big|\big((z-B')^{-1}f_0\big)(x_1)\big|+\ln\|(z-A)^{-1}e_1\|_p}{\ln\|(z-A)^{-1}\|^2}\\
   &\geq \limsup_{j\rightarrow \infty}\frac{\ln\big|\big((z_j-B')^{-1}f_0\big)(x_1)\big|+\ln\|(z_j-A)^{-1}e_1\|_p}{\ln\|(z_j-A)^{-1}\|^2}
   \geq\frac{1}{2}.
   \end{align*}

{\bf Case 3.} $x_i\neq0, i=1,2$.

Let $x_2=\sum\limits_{n=1}^{\infty}\alpha_n e_n$
and $\widetilde{x}_2=\sum\limits_{n=2}^{\infty}\alpha_n e_n$.
We have
  \begin{align*}
  &-\big((z-B')^{-1}f_0\big)(x_1)(z-A)^{-1}e_1+(z-A)^{-1}x_2\\
  &=\big(-\big((z-B')^{-1}f_0\big)(x_1)+\alpha_1\big)(z-A)^{-1}e_1+(z-A)^{-1}\widetilde{x}_2
  \end{align*}
  and
\begin{align*}
  \|(z-T)^{-1}x\|_p&\geq\big\|\big(\alpha_1-\big((z-B')^{-1}f_0\big)(x_1)\big)(z-A)^{-1}e_1+(z-A)^{-1}\widetilde{x}_2\big\|_p\\
  &\geq\big|\big(\alpha_1-\big((z-B')^{-1}f_0\big)(x_1)\big)\big|\|(z-A)^{-1}e_1\|_p
  -\|(z-A)^{-1}\widetilde{x}_2\|_p.
\end{align*}
  Notice that
  \begin{align*}
  \limsup\limits_{z\rightarrow 0}\big|\big(\alpha_1-\big((z-B')^{-1}f_0\big)(x_1)\big)\big|
  =+\infty.
  \end{align*}
  We can pick a sequence $\{\lambda_j\}_{j=0}^{\infty}\subset\mathbb{C}$ with $\lim\limits_{j\rightarrow\infty}|\lambda_j|=0$ such that
  \begin{equation}\label{BT:Geq1}
  \big|\big(\alpha_1-((\lambda_j-B')^{-1}f_0)(x_1)\big)\big|>1,\quad \forall j\geq 0.
  \end{equation}
  Moreover, since $k_{(A,\widetilde{x}_2)}=1$ and $k_{(A,e_1)}=1$, it follows that
  \begin{align*}
  \ln\Big(\big|&\big(\alpha_1-\big((z-B')^{-1}f_0\big)(x_1)\big)\big|\|(z-A)^{-1}e_1\|_p
  -\|(z-A)^{-1}\widetilde{x}_2\|_p\Big)\\
  &\approx
  \ln\Big(\big|\big(\alpha_1-\big((z-B')^{-1}f_0\big)(x_1)\big)\big|\|(z-A)^{-1}e_1\|_p\Big)
  \end{align*}
  when $z$ tends to $0$. Then combining (\ref{T:Ae1}) and (\ref{BT:Geq1}), we can deduce that
  \begin{align*}
  k_{(T,x)}
  &=\limsup_{z\rightarrow 0}\frac{\ln\|(z-T)^{-1}x\|_p}{\ln\|(z-T)^{-1}\|}\\
  &\approx
  \limsup_{z\rightarrow 0}\frac{\ln\|(z-T)^{-1}x\|_p}{\ln\|(z-A)^{-1}\|^2}\\
  &\geq\limsup_{z\rightarrow 0}\frac{\ln\big|\big(\alpha_1-\big((z-B')^{-1}f_0\big)(x_1)\big)\big|\|(z-A)^{-1}e_1\|_p}
  {\ln\|(z-A)\|^{2}}\\
  &\geq \limsup_{j\rightarrow \infty}\frac{\ln\big|\big(\alpha_1-\big((\lambda_j-B')^{-1}f_0\big)(x_1)\big)\big|\|(\lambda_j-A)^{-1}e_1\|_p}
  {\ln\|(\lambda_j-A)^{-1}\|^2}\geq\frac{1}{2}.
  \end{align*}
In summary, $k_{(T,x)}\geq\frac{1}{2}$ for all non-zero $x \in l^{p}(\mathbb{Z})$.
Hence, $\Lambda(T)\subset[\frac{1}{2},1]$.
\end{proof}

\begin{lem}\label{B:md1}
Let $T$ be as in $(\ref{bmT})$ above. If $\{w_n\}_{n=1}^{\infty}$ is monotone decreasing and $p^{\prime}$-summable for some positive number $p^{\prime}$, then $1\in\Lambda(T)$.
\end{lem}

\begin{proof}
Pick $m\geq 1$ such that the sequence $\{\prod\limits_{j=1}^{m}w_{n+j}\}_{n=0}^{\infty}$ is summable.
Write $\alpha_n=\prod\limits_{j=1}^{m}w_{n+j}$ for $n\geq0$. And
let
$y=\left[\begin{array}{cc}
  \xi_1\\
 \xi_2
\end{array}\right]$,
 where
$$
\xi_1=\left[\begin{array}{cccc}
 \alpha_0\\
  \alpha_1\\
  \alpha_2\\
  \vdots
    \end{array}\right]
    \begin{matrix}
e_0\\
e_{-1}\\
e_{-2}\\
 \vdots
\end{matrix}, \ \
    \xi_2=\left[\begin{array}{cccc}
  0\\
  0\\
  0\\
  \vdots
\end{array}\right]\begin{matrix}
e_1\\
e_{2}\\
e_{3}\\
 \vdots
\end{matrix}
.$$
Then $y\in l^p(\mathbb{Z})$ for any $p\in [1,+\infty)$ and
\[
(z-T)^{-1}y=
\begin{bmatrix}
(z-B)^{-1}\xi_1\\
  -\big((z-B')^{-1}f_0\big)(\xi_1)(z-A)^{-1}e_1+(z-A)^{-1}\xi_2
\end{bmatrix}
.\]
It follows that
$$\|(z-T)^{-1}y\|_p
\geq\big|\big((z-B')^{-1}f_0\big)(\xi_1)\big|\|(z-A)^{-1}e_1\|_p.$$
Notice that $\big|\big((z-B')^{-1}f_0\big)(\xi_1)\big|=\big|f_0\big((z-B)^{-1}\xi_1\big)\big|$.
By (\ref{B:kx=1}), one can observe that
\begin{align*}
\limsup_{z\rightarrow0}\frac{\ln\big|\big((z-B')^{-1}f_0\big)(\xi_1)\big|}{\ln\|(z-B')^{-1}\|}\geq 1.
\end{align*}
Since $\|(z-A)^{-1}\|=\|(z-B)^{-1}\|=\|(z-B')^{-1}\|$, we have
\begin{equation*}
\limsup_{z\rightarrow0}\frac{\ln\big|\big((z-B')^{-1}f_0\big)(\xi_1)\big|}{\ln\|(z-B')^{-1}\|}
=\limsup_{z\rightarrow0}\frac{\ln\big|\big((z-B')^{-1}f_0\big)(\xi_1)\big|}{\ln\|(z-A)^{-1}\|}
\geq 1.
\end{equation*}
Combining (\ref{T:Ae1}), we can obtain that
\begin{equation*}
\begin{aligned}
k_{(T,y)}
&=\limsup_{z\rightarrow 0}\frac{\ln\|(z-T)^{-1}y\|_p}{\ln\|(z-T)^{-1}\|}\\
&\approx\limsup_{z\rightarrow 0}\frac{\ln\|(z-T)^{-1}y\|_p}{\ln\|(z-A)^{-1}\|^2}\\
&\geq\limsup_{z\rightarrow 0}\frac{\ln\big|\big((z-B')^{-1}f_0\big)(\xi_1)\big|+\ln\|(z-A)^{-1}e_1\|_p}{\ln\|(z-A)^{-1}\|^2}
\geq1.
\end{aligned}
\end{equation*}
By the fact $0\leq k_y\leq 1$, it holds that $k_y=1$. Hence $1\in\Lambda(T)$.
\end{proof}

\begin{thm}\label{B:2s1}
Let $T$ be as in $(\ref{bmT})$ above.
If there is a positive number $m_0$ such that $\frac{1}{n+2m_0}\leq w_{n+m_0}\leq\frac{1}{n}$ for $n\geq1$,
then $\Lambda(T)=[\frac{1}{2},1]$.
\end{thm}

\begin{proof}
By Lemmas $\ref{B:AT}$ and $\ref{B:md1}$, we have $\Lambda(T)\subset[\frac{1}{2},1]$ and $1\in \Lambda(T)$.
Next, we need only prove $[\frac{1}{2},1)\subset\Lambda(T)$.

Pick a natural number $m>2m_0+2$, for $0<r\leq1$, set
$x_r=\left[\begin{array}{cc}
  \xi_r\\
 \xi_0
\end{array}\right]$,
 where
$$
\xi_r=\left[\begin{array}{cccc}
 1\\
  r\big((1+m)!w_1\big)^{-1}\\
  r^2\big((2+m)!\prod\limits_{j=1}^2w_j\big)^{-1}\\
  \vdots
    \end{array}\right]
    \begin{matrix}
e_0\\
e_{-1}\\
e_{-2}\\
 \vdots
\end{matrix}, \ \
    \xi_0=\left[\begin{array}{cccc}
  0\\
  0\\
  0\\
  \vdots
\end{array}\right]\begin{matrix}
e_1\\
e_{2}\\
e_{3}\\
 \vdots
\end{matrix}
.$$
By (\ref{pli}), we have $x_r\in l^p(\mathbb{Z})$ for any $p\in[1,+\infty)$. Then
\[
(z-T)^{-1}x_r=
\begin{bmatrix}
(z-B)^{-1}\xi_r\\
-\big((z-B')^{-1}f_0\big)(\xi_r)(z-A)^{-1}e_1+(z-A)^{-1}\xi_0
\end{bmatrix}
.\]
It follows that
$$\|(z-T)^{-1}x_r\|_p=\|(z-B)^{-1}\xi_r\|_p+\big|\big((z-B')^{-1}f_0\big)(\xi_r)\big|\|(z-A)^{-1}e_1\|_p.$$
For $\|(z-B)^{-1}\xi_r\|_p$, by (\ref{xrn}), one can observe that
\begin{align*}
\big|f_0\big((z-B)^{-1}\xi_r\big)\big|
\leq
\|(z-B)^{-1}\xi_r\|_p
\leq\big|f_0\big((z-B)^{-1}\xi_r\big)\big|\|\xi_r\|_p.
\end{align*}
So, when $z$ tends to $0$, we have
$$\ln\|(z-B)^{-1}\xi_r\|_p\approx \ln\big|f_0\big((z-B)^{-1}\xi_r\big)\big|.$$
For $\|(z-A)^{-1}e_1\|_p$, by (\ref{T:Ae1}), we have that
$$\ln\|(z-A)^{-1}e_1\|_p\approx\ln\|(z-A)^{-1}\|.$$
For $\big|\big((z-B')^{-1}f_0\big)(\xi_r)\big|$, by $(\ref{fxr})$, it follows that
\begin{align*}
\big|\big((z-B')^{-1}f_0\big)(\xi_r)\big|
&=\big|f_0\big((z-B)^{-1}\xi_r\big)\big|=\big|\frac{1}{z}+\frac{z^{m-1}}{r^{m}}\big(e^{\frac{r}{z}}-\sum\limits_{n=0}^{m}\frac{1}{n!}\frac{r^n}{z^n}\big)\big|.
\end{align*}
Hence, as $z$ tends to $0$, we have
\begin{align*}
\ln\|(z-T)^{-1}x_r\|_p
&\approx\ln\|(z-A)^{-1}e_1\|_p\big|\big((z-B')^{-1}f_0\big)(\xi_r)\big|.
\end{align*}
Notice that
\begin{align*}
\|(z-A)^{-1}\|&=\|(z-B)^{-1}\|
\geq \frac{1}{\|\xi_1\|_p}\|(z-B)^{-1}\xi_1\|_p\\
&\geq \frac{1}{\|\xi_1\|_p}|f_0((z-B)^{-1}\xi_1)|\\
&=\frac{1}{\|\xi_1\|_p}\big|\frac{1}{z}+z^{m-1}\big(e^{\frac{1}{z}}-\sum\limits_{n=0}^{m}\frac{1}{n!}\frac{1}{z^n}\big)\big|
\end{align*}
and
\begin{align*}
\|(z-A)^{-1}\|
\leq
\frac{1}{|z|}+\sum\limits_{k=1}^{m_0}\frac{\prod\limits_{j=1}^k w_j}{|z|^{k+1}}+\frac{1}{|z|}\big(e^{\frac{1}{|z|}}-\sum\limits_{k=0}^{m_0}\frac{1}{k!|z|^{k}}\big).
\end{align*}
So, as $z$ tends to $0$, we have
$\ln\|(z-A)^{-1}\|\approx \ln e^{\frac{1}{|z|}}$ and
\begin{align*}
\ln\|(z-T)^{-1}x_r\|_p
&\approx\ln e^{\frac{1}{|z|}}+\ln \big|\frac{1}{z}+\frac{z^{m-1}}{r^{m}}\big(e^{\frac{r}{z}}-\sum\limits_{n=0}^{m}\frac{1}{n!}\frac{r^n}{z^n}\big)\big|.
\end{align*}
It follows that
\begin{align*}
k_{(T,x_r)}
&=\limsup_{z\rightarrow 0}\frac{\ln\|(z-T)^{-1}x_r\|_p}{\ln\|(z-T)^{-1}\|}\\
&\approx
\limsup_{z\rightarrow 0}\frac{\ln\|(z-A)^{-1}e_1\|_p\big|\big((z-B')^{-1}f_0\big)(\xi_r)\big|}{\ln\|(z-A)^{-1}\|^2}\\
&\approx \limsup_{z\rightarrow 0}
\frac{\ln e^{\frac{1}{|z|}}+\ln\big|\frac{1}{z}+\frac{z^{m-1}}{r^{m}}\big(e^{\frac{r}{z}}-\sum\limits_{n=0}^{m}\frac{1}{n!}\frac{r^n}{z^n}\big)\big|}
{\ln (e^{\frac{1}{|z|}})^2}\\
&=\limsup_{|z|\rightarrow 0}\frac{\ln e^{\frac{1}{|z|}}+\ln\Big(\frac{1}{|z|}+\frac{|z|^{m-1}}{r^{m}}\big(e^{\frac{r}{|z|}}-\sum\limits_{n=0}^{m}\frac{1}{n!}\frac{r^n}{|z|^n}\big)\Big)}
{\ln (e^{\frac{1}{|z|}})^2}\\
&=\limsup_{|z|\rightarrow 0}
\frac{\ln e^{\frac{1}{|z|}}+\ln e^{\frac{r}{|z|}}\Big(\frac{(e^{\frac{r}{|z|}})^{-1}}{|z|}+\frac{|z|^{m-1}}{r^{m}}\big(1-(e^{\frac{r}{|z|}})^{-1}\sum\limits_{n=0}^{m}\frac{1}{n!}\frac{r^n}{|z|^n}\big)\Big)}
{\ln (e^{\frac{1}{|z|}})^2}\\
&=\frac{1+r}{2}.
\end{align*}
Since $0<r\leq1$, it yields that $(\frac{1}{2},1]\subset\Lambda(T)$.

Finally, we prove that $\frac{1}{2}\in\Lambda(T)$.
Suppose that
$\widetilde{x}=\left[\begin{array}{cc}
  \beta_1\\
 \beta_2
\end{array}\right]$,
 where
$$
\beta_1=\left[\begin{array}{cccc}
 1\\
  0\\
  0\\
  \vdots
    \end{array}\right]
    \begin{matrix}
e_0\\
e_{-1}\\
e_{-2}\\
 \vdots
\end{matrix}, \ \
    \beta_2=\left[\begin{array}{cccc}
  0\\
  0\\
  0\\
  \vdots
\end{array}\right]\begin{matrix}
e_1\\
e_{2}\\
e_{3}\\
 \vdots
\end{matrix}
.$$
Then $\widetilde{x}\in l^p(\mathbb{Z})$ and
\[
(z-T)^{-1}\widetilde{x}=
\begin{bmatrix}
 (z-B)^{-1}\beta_1\\
  -\big((z-B')^{-1}f_0\big)(\beta_1)(z-A)^{-1}e_1+(z-A)^{-1}\beta_2\\
\end{bmatrix}
.\]
It follows that
$$\|(z-T)^{-1}\widetilde{x}\|_p=\|(z-B)^{-1}\beta_1\|_p+\big|\big((z-B')^{-1}f_0\big)(\beta_1)\big|\|(z-A)^{-1}e_1\|_p.$$
Notice that
$$\big|\big((z-B')^{-1}f_0\big)(\beta_1)\big|=\big|f_0((z-B)^{-1}\beta_1)\big|=\frac{1}{|z|}$$
and $\|(z-B)^{-1}\beta_1\|_p=\frac{1}{|z|^p}$.
So, when $z$ tends to $0$, we have
\begin{align*}
\ln\|(z-T)^{-1}\widetilde{x}\|_p
&\approx\ln\big|\big((z-B')^{-1}f_0\big)(\beta_1)\big|\|(z-A)^{-1}e_1\|_p\\
&\approx\ln\frac{1}{|z|}+\ln\|(z-A)^{-1}\|.
\end{align*}
Then
\begin{align*}
k_{(T,\widetilde{x})}
&=\limsup_{z\rightarrow 0}\frac{\ln\|(z-T)^{-1}\widetilde{x}\|_p}{\ln\|(z-T)^{-1}\|}\\
&\approx
\limsup_{z\rightarrow 0}\frac{\ln\frac{1}{|z|}+\ln\|(z-A)^{-1}\|}{\ln\|(z-A)^{-1}\|^2}
=\frac{1}{2}.
\end{align*}
Therefore, $\frac{1}{2}\in\Lambda(T)$. In summary $\Lambda(T)=[\frac{1}{2},1]$, completing the proof.
\end{proof}

We close this paper with an interesting question.
\begin{ques}
Given a weighted shift $T$ on $l^\infty$, is there the same power set as $T$ on $\ell^p$ for $p<\infty$?
\end{ques}


\end{document}